\documentclass[11pt]{amsart}

\usepackage{xspace}

\usepackage{amsmath}
\usepackage{amsthm}
\usepackage{amssymb}

 \usepackage{xcolor}

 \usepackage{url}
 \usepackage{graphicx} 

\allowdisplaybreaks 

\newtheorem{theorem}{Theorem}[section]
\newtheorem{prop}[theorem]{Proposition}
\newtheorem{lemma}[theorem]{Lemma}
\newtheorem{cor}[theorem]{Corollary}

\theoremstyle{definition}
\newtheorem{defi}[theorem]{Definition}
\newtheorem{example}[theorem]{Example}

\theoremstyle{remark}
\newtheorem{remark}{Remark}

\newcommand{\K}{\mathcal{K}}
\newcommand{\cat}{\mathop{\mathrm{cat}}}
\newcommand{\gcat}{\mathop{\mathrm{gcat}}}
\newcommand{\Cat}{\mathop{\mathrm{Cat}}}

\newcommand{\sd}{\mathop{\mathrm{sd}}}

\newcommand{\ctg}{\sim_c}   
\newcommand{\ctgcl}{\sim}   

\newcommand{\id}{\mathop{\mathrm{id}}}

\newcommand{\scat}{\mathop{\mathrm{scat}}}
\newcommand{\gscat}{\mathop{\mathrm{gscat}}}
\newcommand{\op}{^\mathrm{op}}

\newcommand{\seco}{\searrow\!\!\!\searrow }     

\newcommand{\dcat}{\mathop{\mathrm{dcat}}}

\newcommand{\thmref}[1]{Theorem~\ref{#1}}
\newcommand{\propref}[1]{Proposition~\ref{#1}}
\newcommand{\lemref}[1]{Lemma~\ref{#1}}
\newcommand{\corref}[1]{Corollary~\ref{#1}}

\newcommand{\examref}[1]{Example~\ref{#1}}
\newcommand{\defiref}[1]{Definition~\ref{#1}}

\newcommand{\figref}[1]{Figure~\ref{#1}}

\begin{document}
\title[LS-category of simplicial complexes and finite spaces]{Lusternik-Schnirelmann category of \\simplicial complexes and finite spaces}

\author{D.~Fern\'andez-Ternero \and E.~Mac\'{\i}as-Virg\'os \and J.A.~Vilches}

\thanks{The first and the third authors are partially supported by PAIDI Research Groups FQM-326 and FQM-189. The second  author was partially supported by MINECO Spain Research Project MTM2013-41768-P and FEDER}

\date{\today}

\begin{abstract}
In this paper we establish a natural definition of
Lus\-ter\-nik-Sch\-nirel\-mann category for simplicial complexes
via the well known  notion of contiguity. This category has the
property of being homotopy invariant under strong equivalences,
and {it} only depends on the simplicial structure rather than
its geometric realization.

In a similar way to the classical case, we also develop a notion of
geometric category for simplicial complexes. We prove that the maximum value over the homotopy class of a given complex is attained in the core of the complex.

Finally, by means of well known relations between simplicial
complexes and posets, specific new results for the topological
notion of {LS-}category are obtained in the setting of finite
topological spaces.
\end{abstract}



\maketitle

\tableofcontents

\section{Introduction}
Lusternik-Schnirelmann category was originally introduced as a
tool for variational problems on manifolds. Nowadays it has been
reformulated as a numerical invariant of topological spaces and
has become an important notion in homotopy theory  and  many other
areas \cite{CORNEA}, as well as in applications like topological
robotics \cite{FARBERGHRIST}. Many papers have appeared on this
topic  and the original definition has been generalized in a
number of different ways. For simplicial complexes and simplicial
maps, the notion of {\em contiguity} is considered as the discrete
version of homotopy. However, although these notions are classical
ones, the corresponding theory of LS-category is missing in the
literature. This paper can be considered as a first step in this
direction.

Still more important, finite simplicial complexes play a
fundamental role in the so-called theory of poset topology, which
connects combinatorics to many other branches of {M}athematics
\cite{KOPPERMAN,WACHS}. Being more precise, such theory allows us
to establish relations between simplicial complexes and finite
topological spaces. On the one hand, finite $T_0$-spaces and
finite partially ordered sets are equivalent categories (notice
that any finite space is homotopically equivalent to a
$T_0$-space). On the other hand, given a finite topological space
$X$ there exists the associated simplicial complex $\K(X)$, where
the simplices are its non-empty chains; and, conversely, given a
finite simplicial complex $K$ there is a finite space $\chi(K)$,
the poset of simplices of $K$, such that $\K(\chi(K))=\sd K$,  the
first barycentric subdivision of $K$. By these constructions we
can see posets and simplicial complexes as essentially equivalent
objects.

In this work we introduce a natural notion of LS-category $\scat
K$ for {any} simplicial {complex $K$}. Unlike other
topological notions established for the geometric realization of
the complex, our approach is directly based on the simplicial
structure. In this context, {\em contiguity classes} are the
combinatorial analogues of homotopy classes.  For instance,
different simplicial approximations to the same continuous map are
contiguous {and the} geometric realizations of contiguous maps
are homotopic.

{Analogously} to the {topological} setting, it is desirable
that this notion of category be a homotopy invariant. In order to
obtain this goal, the notion of strong collapse introduced by
Minian and Barmak \cite{BARMAK} is used instead of the classical
notion of collapse. The existence of {\em cores} or {\em minimal}
complexes is a fundamental difference between strong homotopy
types and simple homotopy types. A simplicial complex can collapse
to non-isomorphic subcomplexes.   However if a complex $K$
strongly collapses to a  minimal complex $K_0$, it must be unique,
up to isomorphism. We prove the homotopical invariance of
simplicial category, and, in particular, that   $\scat K= \scat
K_0$.

In addition, a notion of {\em geometric category} $\gscat K$ is
introduced in the simplicial context. For topological spaces
geometric category is not a homotopical invariant, so  it is
customary to consider  the minimum value of  $\gcat Y$, for all
spaces $Y$ of the same homotopy type as $X$. This process leads to
a homotopical invariant, $\Cat X$, first introduced by Ganea
\cite{CORNEA}. In the simplicial context we prove {several}
results about the behaviour of $\gscat K$ under strong collapses.
Other authors \cite{SCOVILLE} have considered a notion of
geometric category for simple collapses. The essential difference
is that for $\gcat$ one can consider not only the minimum value in
the homotopy class, but also the maximum, which coincides with the
category of the {\em core} of the complex.

By means of the equivalence between simplicial complexes and
finite topological spaces, we get a notion of LS-category of
finite spaces which corresponds with the classical notion, because
the concept of strong homotopy equivalence in the simplicial
context corresponds to the notion of homotopy equivalence in the
setting of finite spaces. Under this point of view new results are
obtained which do not have analogues in the continuous case.


The paper is organized as follows. We start by introducing in
Section~\ref{PRELIM} the basic notions and results concerning the
{link} between simplicial complexes and finite topological spaces,
as well as the definition of classical LS-category. Section 3 is
focused on the study of the simplicial LS-category $\scat K$ of a
simplicial complex $K$. We prove that this notion is a homotopy
invariant, that is, two strongly equivalent complexes have the
same category. The corresponding notion of geometrical category
$\gcat K$ for a  simplicial complex $K$ is studied in Section 4.
We obtain that the geometrical category increases under strong
collapses, and that the maximum value is obtained for the core
$K_0$ of the complex.  Section 5 contains a study on the
LS-category of finite topological spaces. Notice that it is not
the LS-category of the geometric realization $\vert \K(X)\vert$ of
the associated simplicial complex, but it is the category of the
topological space $X$ itself. We have not found any specific study
of LS-category for finite topological spaces {in the literature}.
For instance, we prove that the number of maximal elements minus
one is an upper bound of the category of a finite topological
space. By analogy with the LS-category of simplicial complexes we
establish other results for finite spaces. For instance, we prove
that geometrical category increases when a beat point is erased.
In particular, we exhibit a new example showing that geometrical
category is not a homotopy invariant. This example was
communicated to the authors by J. Barmak and G. Minian. Finally,
in Section 6 we prove that both the category and the geometrical
category decrease when applying the functors $\K$ and $\chi$.

\section{Preliminaries}\label{PRELIM}
\subsection{Simplicial complexes}\label{PRELIMSC}
We recall the notions of contiguity and strong collapse. Let $K,L$
be two simplicial complexes.  Two simplicial maps \linebreak
$\varphi,\psi\colon K\to L$ are {\em contiguous}
\cite[p.~130]{SPANIER} if, for any simplex $\sigma\in K$, the set
$\varphi(\sigma)\cup\psi(\sigma)$ is a simplex of $L${;} that is,
if $v_0,\dots,v_k$ are the vertices of $\sigma$ then the vertices
$f(v_0),\dots,f(v_k),g(v_0),\dots,g(v_k)$ span a simplex of $L$.
{This relation, {denoted by $\varphi \ctg \psi$}, is reflexive and
symmetric,  but in general it is not transitive.}

\begin{defi}Two simplicial maps $\varphi,\psi\colon K \to L$ are
in the same {\em contiguity class}, denoted by $\varphi\ctgcl
\psi$, if there is a sequence {$$\varphi=\varphi_0 \ctg
\cdots \ctg \varphi_n=\psi$$} of contiguous {simplicial}  maps
$\varphi_i\colon K \to L$, $0\leq i\leq n$.
\end{defi}

A simplicial map $\varphi\colon K \to L$ is a {\em strong
equivalence} if there exists $\psi\colon L \to K$ such that
$\psi\circ\varphi \ctgcl \id_K$ and $\varphi\circ \psi \ctgcl
\id_L$. We write $K\ctgcl L$ if there is a strong equivalence
between the {complexes $K$ and $L$}. In the nice paper
\cite{BARMAKMINIAN} Barmak and Minian showed that strong homotopy
types can be described by {certain types of} elementary moves
called {\em strong collapses}. {A detailed exposition is in
Barmak's book \cite{BARMAK}.} {These moves are a particular
case of the well known notion of simplicial collapse}
\cite{COHEN}.

 \begin{defi}\label{DOMIN}A vertex $v$ of a simplicial complex $K$ is {\em dominated} by another vertex $v^\prime\neq v$ if  every maximal simplex that contains $v$ also contains $v^\prime$.
 \end{defi}
If $v$ is dominated by $v^\prime$ then   the inclusion $i \colon
K \setminus v \subset K$ is a strong equivalence. Its homotopical
inverse is the retraction $r\colon K \to K \setminus v$ which is
the identity on $K\setminus v$ and such that $r(v)=v^\prime$. This
retraction is called an {\em elementary strong collapse} from $K$
to $K \setminus v$, denoted by $K\seco K \setminus v$.

A {\em strong collapse} is a finite sequence of elementary
collapses. The inverse of a strong collapse is called a strong
expansion and two complexes $K$ and $L$ have the same {\em strong
homotopy type} if there is a sequence of strong collapses and
strong expansions that transform $K$ into $L$.

\begin{example}
\figref{FigureD} is an example of elementary strong collapse.

\begin{figure}[htbp]
\begin{center}
\includegraphics[width=0.5\textwidth]{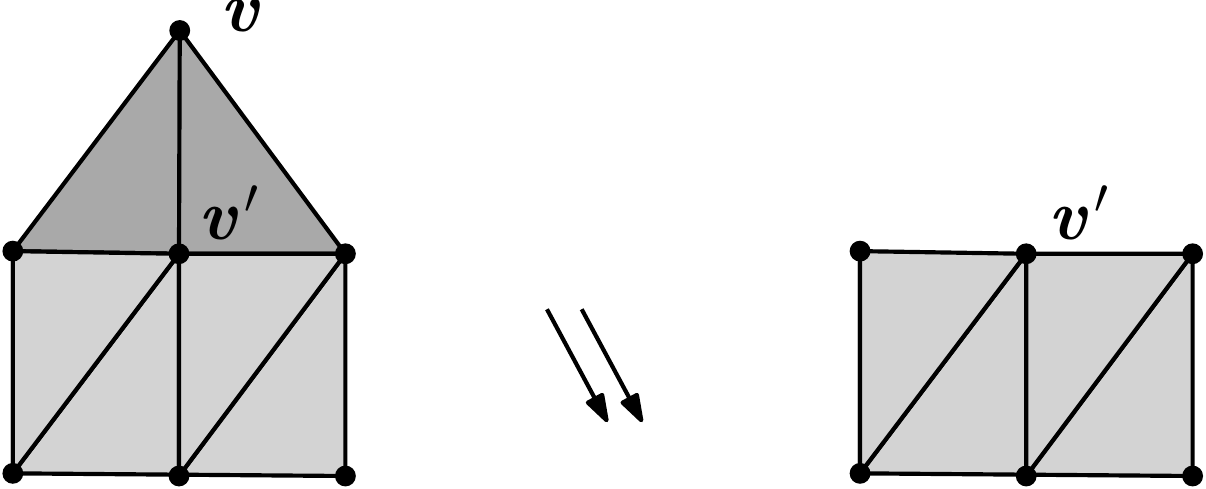}
\end{center}
\caption{Elementary strong collapse} \label{FigureD}
\end{figure}
\end{example}

The following result states that the notions of strong homotopy
type and strong equivalence (via contiguity) are the same.

\begin{theorem}\cite[Cor. 2.12]{BARMAKMINIAN} Two complexes $K$ and $L$ have the same strong homotopy type if and only if $K \ctgcl L$.
\end{theorem}

\subsection{Finite topological spaces}We are interested in {homotopical} properties of finite topological spaces.
First, let us recall the {correspondence}  between finite posets and finite $T_0$-spaces. If $(X,\leq)$ is a partially ordered finite  set, we consider  the  $T_0$ topology {on $X$} given by the basis $\{U_x\}_{x\in X}$ where
$$U_x=\{y\in X \colon y\leq x\}.$$

Conversely, if $(X,\tau)$ is a finite topological space, let $U_x$
be the minimal open set containing {$x\in X$}. Then we can define
a preorder by saying  $x\leq y$ if and only if $U_x\subset U_y$.
This preorder is an order if and only if $\tau$ is $T_0$. Under
this correspondence, a map $f\colon X \to Y$ {between finite
$T_0$-spaces} is continuous {if and only if} $f$ is order
preserving. Order spaces are also called ``Alexandrov spaces''.

\begin{prop}Any (finite) topological space has the homotopy type of a (finite) $T_0$-space.
\end{prop}

\begin{proof}{Take the quotient} by the equivalence relation: {$x\sim y$ if and only if $U_x=U_y$.}
\end{proof}

{From now on, we shall deal with  {finite} spaces which are $T_0$.}

\begin{prop}\cite{WOFSEY} The connected components of $X$ are the equivalence classes of the equivalence relation generated by the order.
\end{prop}

{We now} consider the notion of homotopy.  Let $f,g\colon X\to Y$ be two continuous maps between finite spaces. {We write $f\leq g$  if $f(x)\leq g(x)$ for all $x\in X$. }

\begin{prop}\cite{BARMAKMINIAN}\label{HOMOTOPY} Two maps $f,g\colon X \to Y$ between finite spaces are homotopic, denoted by $f\simeq g$, if and only if they are in the same class of the equivalence relation generated by the relation {$\leq$} between maps.
\end{prop}

\begin{cor}\label{POINT}The basic open sets $U_x\subset X$ are contractible.
\end{cor}

\begin{proof}Since {the point} $x$ is a maximum of $U_x$, it is a deformation retract of $U_x$ by means of the constant map $r=x\colon U_x \to U_x$.
\end{proof}

\subsection{Associated spaces and complexes}\label{ASSOC}
To each finite poset $X$ there is associated the so-called {\em
order complex} $\K(X)$. It   is the simplicial complex with vertex
set $X$ and whose simplices are given by the finite non-empty
chains in the order on $X$. {Moreover}, if $f\colon X \to Y$ is a
continuous map, the associated simplicial map $\K(f)\colon \K(X)
\to \K(Y)$ is defined as $\K(f)(x)=f(x)$ for each vertex $x\in X$.

\begin{prop}\cite[Prop.~2.1.2, Th.~5.2.1]{BARMAK}\label{BARMAK212}
\begin{enumerate}
\item If $f,g\colon X \to Y$ are homotopic maps then the simplicial maps $\K(f)$ and $\K(g)$ are in the same contiguity class.
\item
If two $T_0$-spaces $X,Y$ are homotopy equivalent, then the complexes $\K(X)$ and $\K(Y)$ have the same strong homotopy type.
\end{enumerate}
\end{prop}
Notice that {the reciprocal statements}  are not necessarily true, because two non-homotopic maps $f,g$ may induce maps $\K(f),\K(g)$ which are in the same contiguity class, through simplicial maps which do not preserve order.

Conversely, {it is possible to assign to any finite simplicial complex $K$} its {\em Hasse diagram} {or face poset, that is,} the poset of simplices of $K$ ordered by inclusion. If $\varphi\colon K \to L$ is a simplicial map, the associated continuous map  $\chi(\varphi)\colon \chi(K)\to \chi(L)$ {is given} by $\chi(\varphi)(\sigma)=\varphi(\sigma)$, for any simplex $\sigma$ of $K$.

\begin{prop}\cite[Prop.~2.1.3, Th.5.2.1]{BARMAK}\label{BARMAK213}
\begin{enumerate}
\item If the simplicial maps $\varphi,\psi\colon K \to L$ are in the same contiguity class then the continuous maps $\chi(\varphi),\chi(\psi)$ are homotopic.
\item
If two finite simplicial complexes $K,L$ have the same strong homotopy type, then the associated spaces $\chi(K),\chi(L)$ are homotopy equivalent.
\end{enumerate}
\end{prop}
\subsection{LS-category}\label{LSCAT} {We recall the basic definitions of  Lus\-ter\-nik-Schni\-rel\-mann theory. Well known references are \cite{CORNEA} and \cite{JAMES}.}

An {open subset {$U$ of a topological space $X$ is called {\em categorical} if $U$ can be contracted to a point inside the ambient space $X$. In other words, the inclusion $U\subset X$ is homotopic to some constant map.}

\begin{defi}\label{DEFLSCAT}
The {\em Lusternik-Schnirelmann category}, $\cat X$, of  $X$ is  the least integer $n\geq 0$ such that there is a cover of X by $n+1$ categorical open subsets. We write $\cat X=\infty$ if such a cover does not exist.
\end{defi}

Category is an invariant of homotopy type. {Another interesting
notion, the {\em geometric category}, denoted by $\gcat X$, can be
defined in a similar way  using subsets of $X$ which are
contractible in themselves}, instead of contractible in the
ambient space $X$.  By definition, $\cat X \leq \gcat X$. However,
geometric category is not a homotopy invariant
\cite[p.~79]{CORNEA}.
 \begin{remark}
{For ANRs
one can use  {\em closed} covers, instead of  open
covers, in the definition of LS-category. However, these two notions  would  lead  to different theories in the setting of finite spaces.} {For instance,    for the finite space  of   \examref{Xop}, we obtain different values
 for the corresponding categories.} {This work is limited} to the nowadays most common definition of LS-category, {that is,} using categorical open subsets.
\end{remark}

 \begin{remark}
{Actually, the definition of LS-category by covers is not
well-suited for many constructions in homotopy theory.} This led
to  alternative definitions (Ganea, Whithead \cite{CORNEA}) which
are well known in algebraic topology. However, those constructions
require that the space $X$ satisfies some additional properties.
One of them, the existence of  {\em non-degenerate base-points} is
guaranteed by  Prop. \ref{POINT}. But other properties, like being
Hausdorff or even normal, are not satisfied by finite spaces
{(notice that every  finite $T_1$-space is discrete)},  so we have
not explored them further.
 \end{remark}

\section{LS-category of simplicial complexes}\label{LSCATSC}
 We work in the category of {\em finite} simplicial complexes and simplicial maps \cite{SPANIER}. {{The key notion introduced in this paper is that of}  {\em LS-category} in the simplicial
 setting. {This construction} is the natural one when the notion of ``homotopy'' is that of contiguity class. Contiguous maps were considered in Subsection~\ref{PRELIMSC}.}

 \subsection{Simplicial category}

 \begin{defi} {Let $K$ be a simplicial complex}. We say that the subcomplex $U\subset K$ is {\em categorical} if there exists a vertex $v\in K$ such that the inclusion $i\colon U \to K$ and the constant map $c_v\colon U \to K$ are in the same contiguity class, $i_U\ctgcl c_v$.
 \end{defi}

 In other words, $i$ factors through $v$ up to ``homotopy'' (in the sense of contiguity class). {Notice that a categorical subcomplex may not  be connected.}

 \begin{defi}\label{CATSC}The {\em simplicial LS-category}, $\scat K$, of the simplicial complex $K$, is the least integer $m\geq 0$ such that $K$ can be covered by $m+1$ categorical subcomplexes.
 \end{defi}

 {For instance,  $\scat\ K=0$  if and only if $K$ has the  strong homotopy type of a point. }

 \begin{example} {The simplicial complex $K$ of \figref{FigureE} appears in \cite{BARMAKMINIAN}.}
It is collapsible {(in the usual sense)} but not strongly
collapsible, so $\scat K\geq 1$. We can obtain a cover by two
strongly collapsible subcomplexes taking a non internal 2-simplex
$\sigma$ and its complement $K\setminus\sigma$. {Thus $\scat
K=1$.}

\begin{figure}[htbp]
\begin{center}
\includegraphics[width=0.3\textwidth]{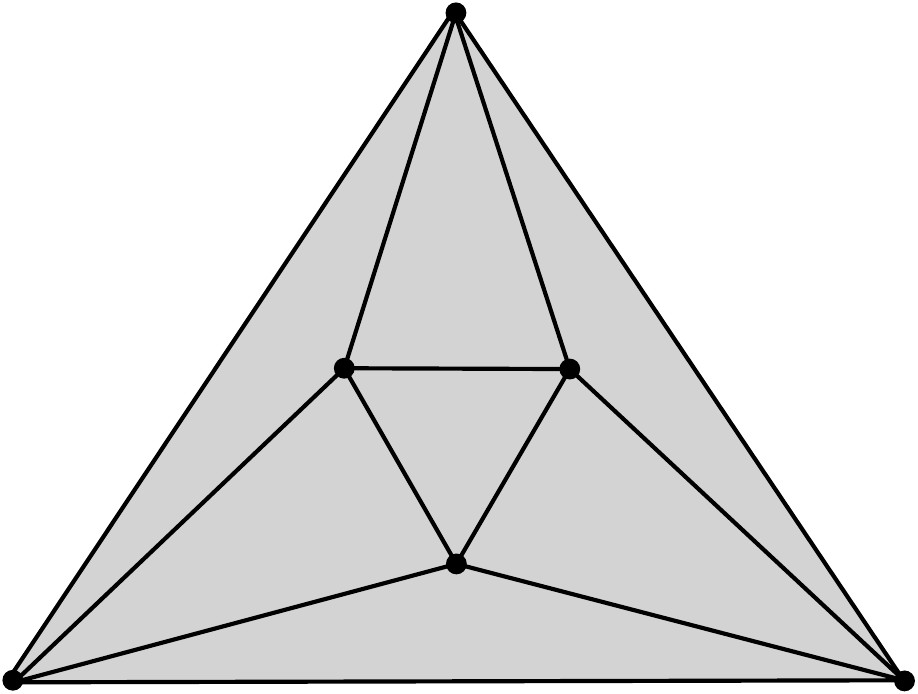}
\end{center}
\caption{A complex $K$ with $\scat K=1$.}
\label{FigureE}
\end{figure}

{This example shows that scat depends on the simplicial
structure more than on the geometric realization of the complex.}
 \end{example}

 \subsection{Homotopical invariance}

{The most important property of the simplicial category is
that it is an invariant of the strong equivalence type, as we
shall prove now.}

 \begin{theorem}\label {MAIN} Let $K\sim L$ {be} two strongly equivalent complexes. {Then} $\scat K =\scat L$.
 \end{theorem}

 We begin with two Lemmas which are easy to prove.

 \begin{lemma}\label{UNO}Let $f,g\colon K \to L$ be {two } contiguous maps, $f\ctg g$, and let $i\colon N\to K$
 (resp.  $r\colon L\to N$) {be another simplicial map}.
  Then $f\circ i\ctg g\circ i$ (resp. $r\circ f \ctg r\circ g$).
 \end{lemma}

 \begin{lemma}\label{DOS}Let
$$K=K_0 \stackrel{f_1}{\to} K_1\to \cdots \stackrel{f_n}{\to} K_n=L$$
 and
$$L=K_n \stackrel{g_n}{\to} \cdots \to  K_1 \stackrel{g_1}{\to}K_0=K$$
be two sequences of maps such that $g_i\circ f_i \ctg 1$ and
$f_i\circ g_i\ctg 1$, for all $i\in\{1,\dots,n\}$. Then the
complexes $K$ and $L$ are strongly equivalent, $K\ctgcl L$.
 \end{lemma}

The main Theorem \ref{MAIN} will be {a direct}  consequence of the following Proposition (by interchanging the roles of $K$ and $L$).

\begin{prop} Let $f\colon K \to L$ and $g\colon L\to K$ be simplicial maps such that $g\circ f \ctgcl 1_K$. Then $\scat K \leq \scat L$.
\end{prop}

\begin{proof}Let $U\subset L$ be a categorical subcomplex. Since the inclusion $i_U$ is in the  contiguity class of some constant map $c_v$, there exists a sequence of maps $\varphi_i\colon U \to K$, $0\leq i\leq n$, such that
{$$i_U=\varphi_0  \ctg \cdots \ctg \varphi_n=c_v.$$} Take the subcomplex $f^{-1}(U)\subset K$. We shall prove that $f^{-1}(U)$ is categorical.
Since $g\circ f\sim 1_K$, there is a sequence of maps $\psi_i\colon K\to K$, $0\leq i\leq {m}$,  such that
$$1_K=\psi_0 \ctg\cdots\ctg\psi_m=g\circ f.$$
Denote {by} $f^\prime$ the restriction of $f$ to $f^{-1}(U)$, with values in $U$, that is, $f^\prime\colon f^{-1}(U)\to U$, defined by $f^\prime(x)=f(x)$. Denote by $j\colon f^{-1}(U)\subset K$ the inclusion.
Then:
{\begin{equation}\label{PRIM}
j=1_K\circ j=\psi_0\circ j\ctg \cdots \ctg \psi_m\circ j = g\circ f\circ j
\end{equation}}
by Lemma \ref{UNO}. {Since $f\circ j=i_U\circ f^\prime $, we
have} {\begin{equation}\label{SECON} g\circ f\circ j=g\circ
i_U\circ f^\prime=g\circ\varphi_0\circ f^\prime  \ctg \cdots \ctg
g\circ \varphi_n\circ f^\prime.
\end{equation}}
But $\varphi_n=c_v$, so $g\circ \varphi_n\circ f^\prime\colon
f^{-1}(U) \to g(U)$ is the constant map $c_{g(v)}$. Combining
\eqref{PRIM} and \eqref{SECON} we obtain
$$j\ctgcl c_{g(v)}.$$
Therefore, the subcomplex $f^{-1}(U)\subset K$ is categorical.

Finally, let {$k=\scat L$} {and} let $\{U_0,\dots,U_k\}$ be a categorical cover of $L$;
then $\{f^{-1}(U_0),\dots, f^{-1}(U_k)\}$ is a categorical cover of $K$, which shows that $\scat K\leq k$.
\end{proof}

{A {\em core} of a finite simplicial complex $K$  is a
subcomplex $K_0\subset K$ without dominated vertices, such that $K\seco K_0$ \cite{BARMAKMINIAN}.
Every complex has a core, which is unique up to isomorphism,
and two finite simplicial complexes have the same strong homotopy type
if and only if their cores are isomorphic. Since $\scat$ is an invariant of the strong homotopy type (\thmref{MAIN}) we have proved the following result.

\begin{cor}\label{CORK0} Let $K_0$ be the core of the simplicial complex $K$. Then $\scat K = \scat K_0$.
\end{cor}}

\section{Geometric category}
{As in the classical case}, {we shall   introduce a notion  of
{\em simplicial geometric category} $\gscat$ in the simplicial
setting, when   ``homotopy'' means to be in the same contiguity
class}. Another so-called {\em discrete category}, $\dcat$, which
takes into account the notion of {\em collapsibility} instead of
strong collapsibility, has been considered by Scoville et al in
\cite{SCOVILLE}. But in contrast with the simplicial LS-category
introduced in Section \ref{LSCATSC}, both $\gscat$ and $\dcat$ are
not homotopy invariant. The problem must then be overcome by
taking the infimum of the category values {over} all  simplicial
complexes which are homotopy equivalent to the given one.

However, our geometric category possesses a remarkable property: due to the  notion of {\em core} complex explained before, there is also a {\em maximum} of category among the complexes in a given homotopy class.

\begin{remark}It is possible to do a translation of the notion of simple collapsibility to finite topological spaces, by means of the notion of {\em weak} beat point \cite{BARMAKMINIAN2}.
\end{remark}

\subsection{Simplicial geometric category}

According to the notion of strong collapse (defined in Section
\ref{PRELIM}), a simplicial complex $K$ is {\em strongly
collapsible} if it is strongly equivalent to a point.
Equivalently, the identity $1_K$ is in the contiguity class of
some constant map $c_v\colon K \to K$.

\begin{defi}
{The {\em simplicial  geometric category}} $\gscat K$ of the
simplicial complex $K$ is the least integer $m\geq 0$ such that
$K$ can be covered by $m+1$ strongly collapsible subcomplexes.
That is, there exists {a cover} $U_0,\dots,U_{m}\subset K$ of $K$
such that    $U_i\sim \ast$, for all $i\in\{0,\dots m\}$.
\end{defi}

Notice that strongly collapsible subcomplexes must be connected. \begin{prop}$\scat K\leq \gscat K$.
\end{prop}

\begin{proof} The proof is reduced to checking that a strongly
collapsible subcomplex is categorical: in fact, the {only}
difference is that in the first case the identity $1_U$ is in the
contiguity class of some constant map $c_v$, while in the second
it is the inclusion $i_U\colon U \to X$  that satisfies $i_U\sim
c_v$.
\end{proof}

\subsection{Behaviour under strong collapses} Obviously, $\scat$ and $\gscat$ are invariant by simplicial isomorphisms.  Moreover we proved in \thmref{MAIN} that $\scat$ is a homotopy invariant.
The next Theorem shows that strong collapses increase the
geometric category.

\begin{theorem}If $L$ is a strong collapse of $K$ then $\gscat L\geq \gscat K$.
\end{theorem}
\begin{proof}Without {loss} of generality we may assume that there is an elementary strong collapse $r\colon K\to
L=K \setminus v$ (see Definition \ref{DOMIN}). If $i\colon L
\subset K$ is the inclusion, then $r\circ i=1_L$ while $\sigma
\cup (i\circ r)(\sigma)$ is a simplex of $K$, for any simplex
$\sigma$ of $K$. Let $V$ be a strongly collapsible subcomplex of
$L$, that is, the identity $1_{V}$ is in the contiguity class of
some constant map $c_w\colon V \to V$. That means that there is a
sequence of maps $\varphi_i\colon V \to V$, $0\leq i\leq n$, such
that
$$1_{V}=\varphi_0\ctg \cdots \ctg \varphi_n=c_w.$$
{Let us denote} $r^\prime=r^{-1}(V) \to V$ the restriction of $r$ to $r^{-1}(V)$, with values in $V$.
Analogously denote $i^\prime\colon V \to r^{-1}(V)$ the inclusion (this is well defined because $r\circ i =1_{V}$).

Then, by  \lemref{UNO}, $\varphi_i\ctg \varphi_{i+1}$ implies $i^\prime\circ\varphi_i\circ r^\prime \ctg i^\prime\circ \varphi_{i+1}\circ r^\prime$. Clearly
$$i^\prime\circ \varphi_n\circ r^\prime=i^\prime\circ c_w\circ r^\prime =c_{i(w)}$$
is a constant map. On the other hand it is
$$i^\prime \circ \varphi_0\circ r^\prime=i^\prime\circ 1_{V}\circ r^\prime = i^\prime\circ r^\prime$$
and the latter map is contiguous to $1_{r^{-1}(V)}$. This is true
because if $\sigma$ is a simplex of $r^{-1}(V)$ then it is a
simplex of $K$, so $\sigma\cup (i\circ r)(\sigma)$ is a simplex of
$K$, which is contained in $r^{-1}(V)$ because $r\circ i=1_V$. But
$(i\circ r)(\sigma)=(i^\prime\circ r^\prime)(\sigma)$, so
$\sigma\cup (i^\prime\circ r^\prime)(\sigma)$ is a simplex of
$r^{-1}(V)$.

We have then proved that the constant map $c_w$ is in the same
contiguity class as the identity of $r^{-1}(V)$, which proves that
the latter is strongly collapsible.

Now, let $m=\gscat L$ and $\{V_0,\dots,V_m\}$ a cover of $L$ by
strongly collapsible subcomplexes. Then
$\{r^{-1}(V_0),\dots,r^{-1}(V_m)\}$ is a cover of $K$ by strongly
collapsible subcomplexes. This proves that $\gscat K\leq m$.
\end{proof}

\begin{remark}
Example 5.7 about finite spaces and the relations established in
Section 6 lead us to thinking that the inequality {in the}
previous Theorem is not an equality{. However,}  we have
not found an example of {a} complex simplicial {where} the
inequality is strict.
\end{remark}

Given any finite complex $K$, by successive elimination of
dominated vertices one obtains the core $K_0$ of the complex $K$,
which is the same for all the complexes in the homotopy class of
$K$. Then we have the following result (compare with
\corref{CORK0}.

\begin{cor}The geometric category $\gscat K_0$ of the core $K_0$ of the complex $K$ is the maximum value of $\gscat L$ among all the complexes $L$ which are strongly equivalent to $K$.
\end{cor}}

\section{LS-category of finite spaces}
In this paper finite posets are considered as topological spaces
by themselves, and not as geometrical realizations of its
associated order complexes. That is, as {emphasized} in
\cite[p. 34]{BARMAK}, to say that a finite $T_0$-space X is
contractible is different from saying that $\vert \K(X)\vert$ is
contractible (although   $X$ and $\vert \K(X)\vert$ have the same
{\em weak} homotopy type). In this context we shall consider the
usual notion of {LS-category of topological spaces
\cite{CORNEA}. We have already introduced it in
\defiref{DEFLSCAT}.}

\subsection{Maximal elements}

The following result establishes an upper bound for the category of a finite poset. {Notice that there is not a result of this kind for non-finite topological spaces}

\begin{prop}\label{MAX} Let $M(X)$ be the number of maximal elements of $X$. Then  $\cat X\leq \gcat X<M(X)$.
\end{prop}

\begin{proof}If {$x\in X$} is a maximal element then $U_x$ is contractible (\corref{POINT}), so maximal elements determine  a categorical cover.
\end{proof}

In particular, a space with a maximum is contractible, as it is well known.

{It is also known that if $X$ has a unique {\em minimal} element
$x$ then $X$ is contractible, because the identity is homotopic to
the constant map $c_x$. Even more, a space $X$ is contractible if
and only if its {\em opposite} space $X\op$ (that is, reverse
order) is contractible. However, the {LS-categories} of $X$ and
$X\op$ may not {coincide}, as the following Example shows. This is
a quick way to check that $X$ and $X\op$ are not homotopy
equivalent, even if they always are {\em weak homotopy}
equivalent.}

{\begin{example}\label{Xop} In \figref{FigureA} it is clear
that $\cat X=1$ because $X$ is not contractible and $\cat X<2$ by
\propref{MAX}. However $\cat X\op=2$ since $\cat X\op<3$ and it is
easy to check that the unions of any two open sets $U_{y_i}\cup
U_{y_j}$ are not contractible.

\begin{figure}[htbp]
\begin{center}
\includegraphics[width=0.8\textwidth]{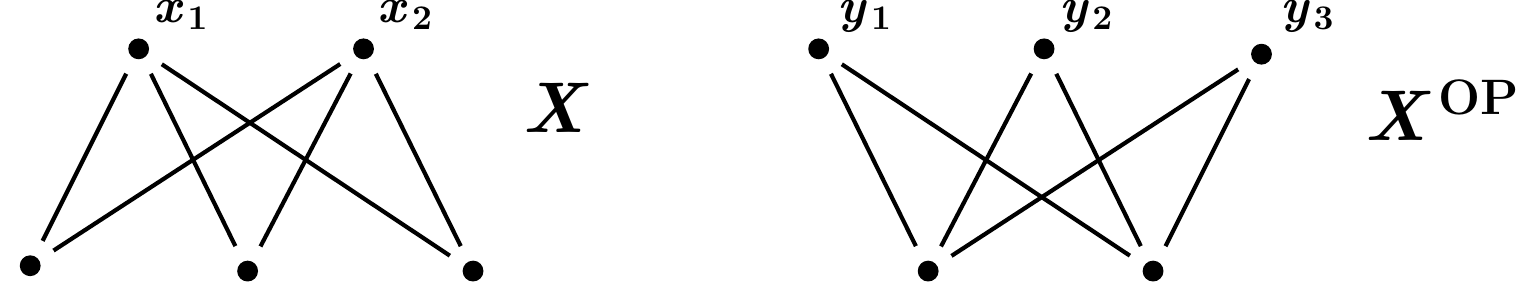}
\end{center}
\caption{A space where $\cat X \neq \cat X\op$.}
\label{FigureA}
\end{figure}
Notice that for any categorical open cover, the open sets $U_x$ corresponding to maxima must be contained in some element of the cover.
\end{example}}

\subsection{Geometric category}As {it was} pointed out in Section \ref{PRELIM}, another homotopy invariant, $\Cat X$, can be defined as the least geometric category of all spaces in the homotopy type of $X$. A peculiarity of finite topological spaces is that it is also possible to consider the {\em maximum} value of $\gcat$  in each  homotopy type. We shall prove that this maximum is attained in the so called  {\em core}  space of $X$, a notion introduced by Stong \cite{STONG}.

The next definition is equivalent to that of linear and collinear
points in  \cite [Th.~2]{STONG}, called {\em beat points} by other
authors \cite{BARMAK,BARMAKMINIAN,MAY}.

\begin{defi}\label{BEAT}Let $X$ be a finite topological space. A point $x_0\in X$ is a {\em beat point} if there exists another point $x_0^\prime\neq x_0$
satisfying the following conditions:
\begin{enumerate}
\item
If $x_0<y$ then $x_0^\prime\leq y$;
\item
if $x<x_0$ then $x\leq x_0^\prime$.
\item
$x_0$ and $x_0^\prime$ are comparable.
\end{enumerate}
\end{defi}

{In other words,  a beat point  covers exactly  one point or  it is covered
by exactly one point. \figref{FigureG} shows a beat point with $x_0\leq x_0^\prime$.}

\begin{figure}[htbp]
\begin{center}
\includegraphics[width=0.35\textwidth]{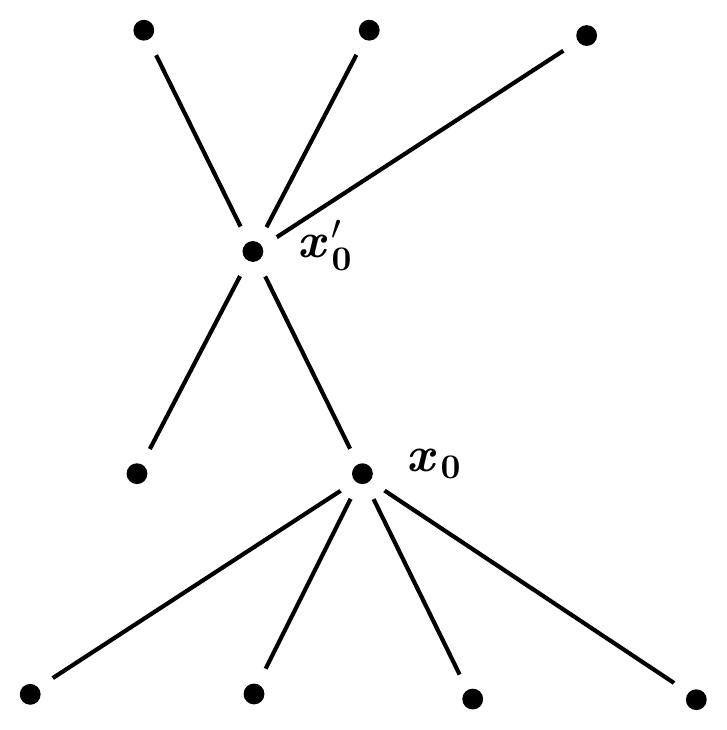}
\end{center}
\caption{An {\em up} beat point $x_0$}
\label{FigureG}
\end{figure}

\begin{prop}\label{RETRACT} If $x_0$ is a   beat point of $X$ then the map $r\colon X\to X\setminus x_0$ given by $r(x)=x$ if $x\neq x_0$ and $r(x_0)=x_0^\prime$, is continuous and satisfies $r\circ i=\id$ and $i\circ r \simeq \id$.
\end{prop}

\begin{cor}\label{OBSERV}If $f\colon X \to X$ is a continuous map such that $f(x_0)=x_0$, then the map $g$ which equals $f$ on $X\setminus x_0$ but  sends $x_0$ onto $x_0^\prime$ is homotopic to $f$.
\end{cor}

Since $X\setminus x_0$ is a deformation retract of $X$
(\propref{RETRACT}) it follows that $\cat X\setminus x_0=\cat X$.
However $\gcat$ is not a homotopical invariant. {The next
Theorem shows that geometrical category increases when a beat
point is erased.}

 \begin{theorem}\label{SUBE}If $x_0$ is a beat point of $X$ then $\gcat X\setminus x_0\geq \gcat X$. \end{theorem}
\begin{proof}Let $U_0,\dots,U_n$ be a cover of $X\setminus x_0$ such that each $U_i$ is an open subset of $X\setminus x_0$, contractible in itself.  We shall define a  cover $U_0^\prime,\dots,U_n^\prime$ of $X$ as follows.

{Let $x_0^\prime$ be a point associated to the beat point $x_0$ as in \defiref{BEAT}.} For each $U_i$, $0\leq i\leq n$, we take:
{\begin{enumerate}
\item
If $x_0$ is a maximal element of $X$ then $x_0^\prime\leq x_0$ and
\begin{enumerate}
\item
there is some $U_i$ which contains $x_0^\prime$, so we take $U_i^\prime=U_i\cup\{x_0\}$;\label{A1}
\item
for the other $U_j$'s, if any, we take $U_j^\prime=U_j$.\label{A2}
\end{enumerate}
\item
If $x_0$ is not a maximal element, then it happens that
\begin{enumerate}
\item
for some of the $U_i$'s  there exists $y\in U_i$ such that $x_0< y$; then we take $U_i^\prime =U_i\cup \{x_0\}$;\label{B1}
\item
for the other $U_j$'s, if any, which satisfy $x<x_0$ for all $x\in
U_j$, we take  $U_j^\prime=U_j$.\label{B2}
\end{enumerate}
\end{enumerate}}
Notice that condition \eqref{B1} implies that $x_0^\prime\in U_i$
because $x_0<y$  implies $x_0^\prime\leq y$ (by definition of beat
point), and $U_i$ is an open subset of $X\setminus x_0$, so the
basic open set $U_y$ is contained in $U_i$.

We shall {verify} that each $U_i^\prime$ is an open subset of
$X$.

Let $y\in U_i^\prime$ and $x\leq y$.  If $x,y\neq x_0$ then $x\in
U_i\subset U_i^\prime$ because $U_i$ is an open subset of
$X\setminus x_0$. In cases \eqref{A1} and \eqref{B1}, if $y=x_0$
and $x<x_0$ then $x\leq x_0^\prime$ by definition of beat point,
and we know that $x_0^\prime\in U_i$, {so we conclude that $x\in
U_i\subset U_i^\prime$. Finally,} if $x=x_0<y$ then  $x\in
U_i^\prime=U_i\cup \{x_0\}$. In cases \eqref{A2} and \eqref{B2},
neither $x_0\in U_i^\prime$ nor $x=x_0<y\in U_i$ are possible.

Moreover it is easy to check that $x_0$ is still a beat point of $U_i^\prime$, with the same  {associated point} $x_0^\prime$.

Let $U=U_i$ for some $i\in\{0,\dots,n\}$; {since $U$ is
strongly collapsible}, the identity map $\id \colon U\to U$ is
homotopic to some constant map $c\colon U\to U$.  {By
\propref{HOMOTOPY}, that means that} there is a sequence
$\id=\varphi_0, \cdots, \varphi_n=c$ of maps $\varphi_k\colon U
\to U$ such that each consecutive pair satisfies either
$\varphi_i\leq \varphi_{i+1}$ or $\varphi_i\geq \varphi_{i+1}$.

We shall {prove} that the identity of $U^\prime=U_i^\prime$ is
homotopic to a cons\-tant map. Obviously it suffices to consider
cases \eqref{A1} and \eqref{B1}. Define $\varphi_k^\prime\colon
U^\prime \to U^\prime$ as follows: $
\varphi_k^\prime(x)=\varphi_k(x)$ if $x\neq x_0$ and $
\varphi_k^\prime(x_0)=\varphi_k(x_0^\prime)$. Thus the maps
$\varphi_k^\prime$ are continuous because $\varphi_k$ preserves
the order, hence $\varphi_k^\prime$ preserves the order too, as it
is easy to check. Moreover if $\varphi_i\leq \varphi_{i+1}$ then
$\varphi_i^\prime\leq \varphi_{i+1}^\prime$ (analogously for
$\geq$). So we have that $\varphi_0^\prime \sim \varphi_n^\prime$.
{Now}, the map $\varphi_n$  is constant,  so it is
$\varphi_n^\prime$. Finally, the map $\varphi_0^\prime$ is not the
identity, but it is homotopic to the identity by \lemref{OBSERV}.

Finally is is easy to check that the open sets $U_0^\prime,
,\dots,U_n^\prime$ form a cover of $X$. Since they are
contractible, it follows that $\gcat X\leq n$.
\end{proof}

 After a finite number of steps,  by successive elimination of all the beat points,
 {a {\em core} or {\em minimal} space $X_0$  is obtained},  which is in the same homotopy class as $X$. It is known that this core space is unique up to homeomorphism {\cite[Th.~4]{STONG}}.

\begin{example}\label{X-COREX}
The example in \figref{FigureH}, communicated to the authors by J.
Barmak and G. Minian, shows that the inequality of \thmref{SUBE}
can be strict.

\begin{figure}[htbp]
\begin{center}
\includegraphics[width=0.8\textwidth]{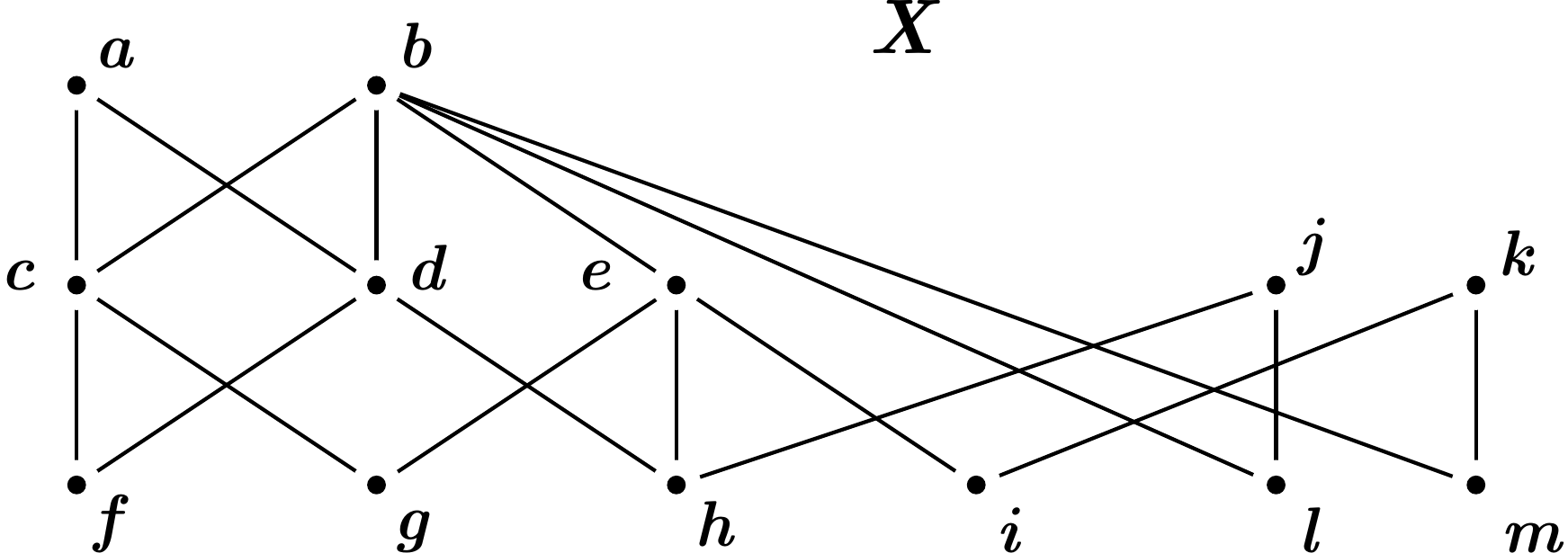}
\caption{A finite space $X$ with $\gcat X_0 > \gcat X$.}
\label{FigureH}
\end{center}
\end{figure}

On the one hand, since $X$ is not contractible, $\gcat X\geq 1$.
In addition,   $\{U_a\cup U_b, U_e\cup U_j\cup U_k\}$ is a cover
of $X$ by open subsets which are contractible in themselves, so we
conclude that $\gcat X=1$.

On the other hand, let us consider the core $X_0$
(\figref{FigureI}) of the finite space $X$.

\begin{figure}[htbp]
\begin{center}
\includegraphics[width=0.5\textwidth]{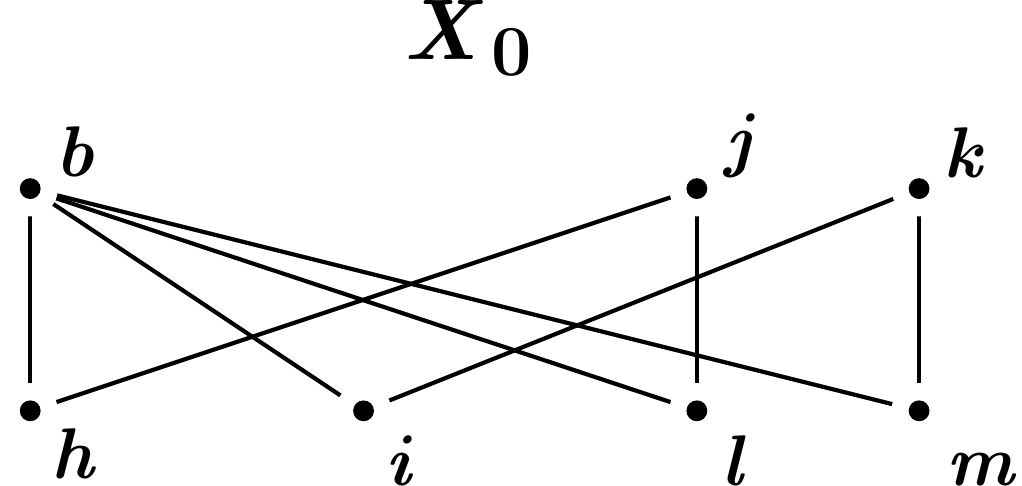}
\caption{The core $X_0$ of the finite space $X$ in
\figref{FigureH}.} \label{FigureI}
\end{center}
\end{figure}

We can observe that $\{U_b, U_j, U_k\}$ is a cover of $X_0$ by
open subsets which are contractible in themselves, so $\gcat
X_0\leq 2$. Finally, we can prove that it is not possible to cover
$X_0$ with just two subsets: since each open subset of the cover
has to be union of the basic open subsets $U_b, U_j, U_k$ and we
note that the unions of two of these open sets are not
contractible, we conclude that there is no cover with two
elements. Thus $\gcat X_0=2$.

Therefore the inequality of  \thmref{SUBE} is strict for this
example.

\end{example}

\begin{remark} Notice that $X_0$ is a new example of a space whose geometrical category does not coincide with its LS-category.  It is also a new example showing that the geometrical category of
topological spaces is not a homotopy invariant. A classical
example is due to Fox \cite{FOX}. Other examples are given by
Clapp and Montejano in \cite{CM} , see also Section 3.3 of
\cite{CORNEA}.
\end{remark}

 {The next Corollary is a consequence of \thmref{SUBE}.}

\begin{cor}The geometric category $\gcat X_0$ of the core space $X_0$ of $X$ equals the maximum of the geometrical categories in its homotopy class.
\end{cor}

\begin{example}\label{X-cat-gcat}
 \figref{FigureB} shows a space where $\cat X=1$ while $\gcat X=2$. {Let us see this.} {Since $X$ is not contractible, $\gcat X\geq \cat X\geq 1$.} Moreover,   $\{U_{x_1}\cup U_{x_4}, U_{x_2}\cup U_{x_3}\}$ is a cover of $X$ by categorical open subsets. {So we conclude that $\cat X=1$.}

 \begin{figure}[htbp]
\begin{center}
\includegraphics[width=0.6\textwidth]{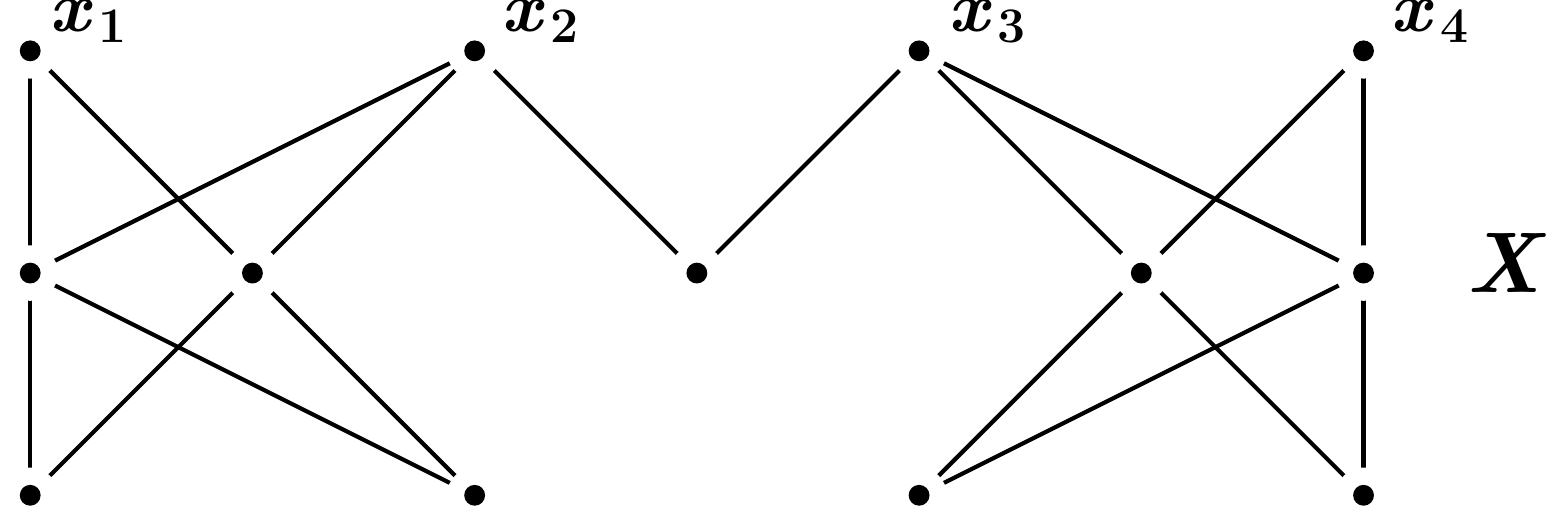}
\end{center}
\caption{A space with $\cat X=1$ but $\gcat X=2$.}
\label{FigureB}
\end{figure}

On the other hand, $\{U_{x_1}, U_{x_2}\cup U_{x_3}, U_{x_4}\}$ is
a cover of $X$ by open subsets which are contractible in
themselves, so $\gcat X\leq 2$. Finally, we can prove that there
is no such kind of cover of $X$ with just two subsets: since each
open subset of the cover  has to be union of basic open subsets
$U_{x_i}$, where $x_i$ are maximal points, and taking into account
that the unique union of $U_{x_i}$'s that is contractible is
$U_{x_2}\cup U_{x_3}$, we conclude that it is not possible to get
a cover with two elements. Thus $\gcat X=2$.

\end{example}

\section{Relation between categories}
We study the relation between the category of a finite $T_0$-poset $X$ and {the simplicial category} of the associated order complex $\K(X)$. {Analogously, a comparison will be done} between the category of a simplicial complex $K$ and its induced Hasse diagram $\chi (K)$. {The corresponding definitions were given in Section \ref{ASSOC}.}

\begin{prop}\label{AAA}{Let $X$ be  a finite poset and $\K(X)$ its associated order complex. Then $\scat \K(X)\leq \cat X$.}
\end{prop}

\begin{proof}Let $U_0,\dots,U_n$ be a categorical cover of $X$. Then the associated simplicial complexes $\K(U_k)$,
$1\leq k\leq n$, cover $\K(X)$. By definition of LS-category of a
topological space (\defiref{LSCAT}), each inclusion $i_k\colon
U_k\subset X$ is homotopic to some constant map $c_k\colon U_i \to
X$, that is, $i_k\simeq c_k$. Then, by \thmref{BARMAK212}, the
simplicial maps $\K(i_k)$ and $K(c_k)$ from $\K(U_k)$ into $\K(X)$
are in the same contiguity class. Clearly $\K(i_k)$ is the
inclusion $\K(U_k)\subset \K(X)$, and $\K(c_k)$ is a constant map.
Thus, by definition of LS-category of a simplicial complex
(\defiref{CATSC}) the {family of} complexes $\K(U_k)$ {
forms} a categorical cover of $\K(X)$ and {thus}  $\scat
\K(X)\leq n$.
\end{proof}

A completely analogous proof gives {the following inequality for the corresponding geometric categories.}

\begin{prop}\label{GSCATX}  $\gscat \K(X)\leq \gcat X$.
\end{prop}

\begin{example}
Let us consider {(\figref{FigureC})} the order complex $\K(X)$ of the finite space $X$ of \examref{X-cat-gcat}.

\begin{figure}[htbp]
\begin{center}
\includegraphics[width=0.5\textwidth]{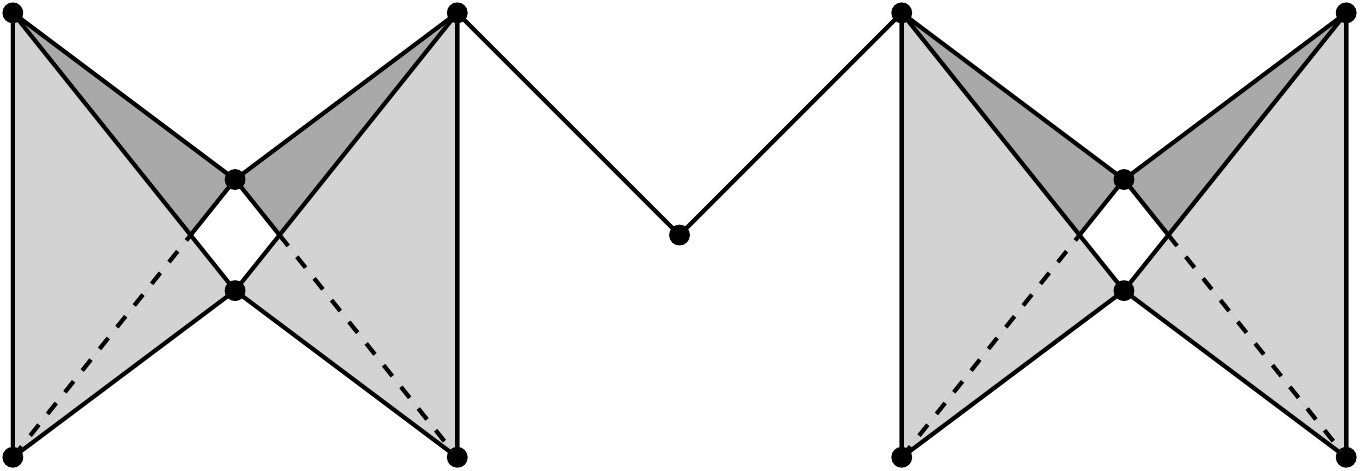}
\caption{The order complex $\K(X)$ of the space $X$ given in \figref{FigureB}.}
\label{FigureC}
\end{center}
\end{figure}

  {Since $\K(X)$ is not strongly collapsible, $\gscat \K(X)\geq
1$. {In addition},   the   {two} strongly collapsible
subcomplexes given in \figref{Figure3132} cover $\K(X)$. So we
conclude that $\gscat \K(X)=1$.

\begin{figure}[htbp]
\begin{center}
\includegraphics[width=0.4\textwidth]{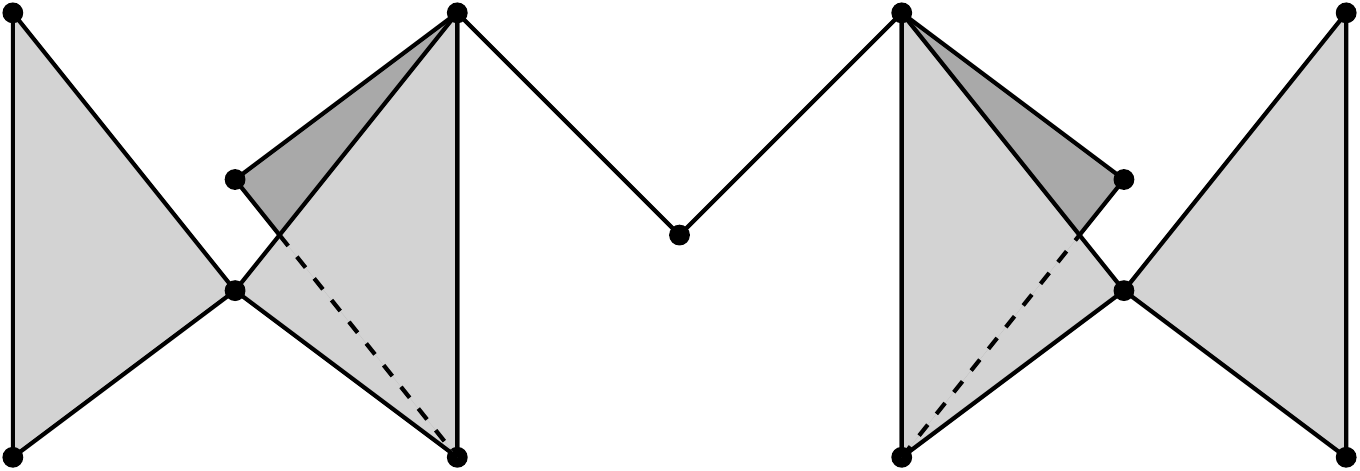}\hspace*{1.5cm}\includegraphics[width=0.4\textwidth]{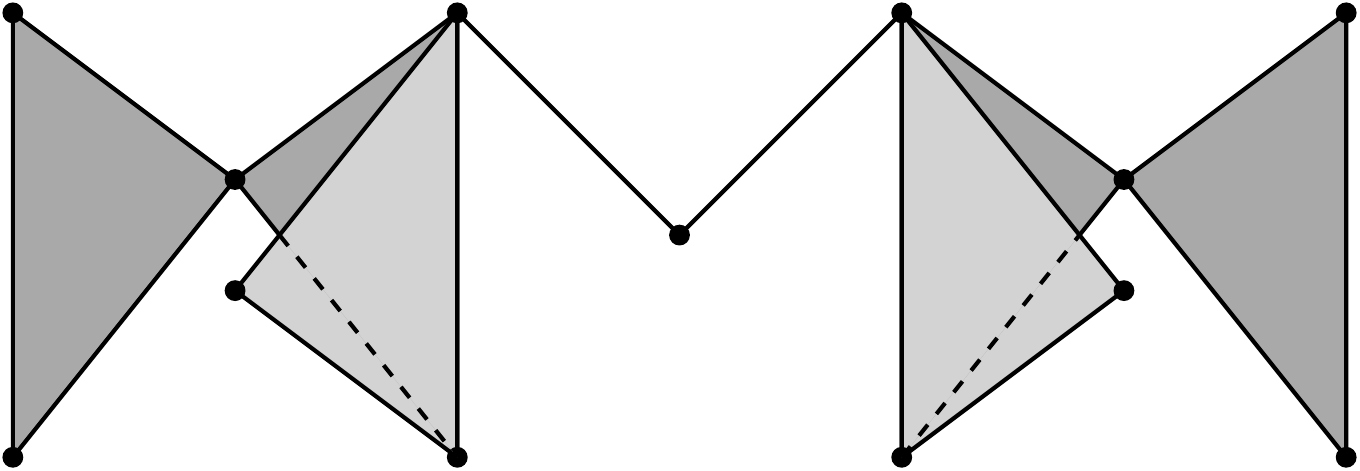}
\caption{Two strongly collapsible subcomplexes}
\label{Figure3132}
\end{center}
\end{figure}

It is interesting to point out that the inequality of  \propref{GSCATX} is strict for this example.
However, the upper bound of the Proposition~\ref{AAA} is {attained} since $\scat \K(X)=\cat X=1$.}
\end{example}

{Now we shall prove analogous inequalities relating the simplicial category of a finite complex $K$ and the topological category of  the face poset  $\chi(K)$.}

\begin{prop}\label{BBB} Let $K$ be a simplicial complex and $\chi(K)$ its Hasse diagram. Then $\cat \chi(K)\leq \scat K$.
\end{prop}

\begin{proof} Let $K_0,\dots,K_n$ be a cover of $K$ by subcomplexes such that each inclusion $i_k\colon K_k\subset K$
is in the same contiguity class of some constant map $c_k\colon
K_k \to K$. Then, {using \propref{BARMAK213}}, the continuous
maps $\chi(i_k)$ and $\chi(c_k)$ are homotopic. By definition
(Section~\ref{ASSOC}), the first one is the inclusion
$\chi(K_k)\subset \chi(K)$, and the second one is a constant map.
Then $\chi(K_0),\dots,\chi(K_n)$ is a categorical cover of
$\chi(K)$. Thus $\cat \chi(K) \leq n$.
\end{proof}

A completely analogous proof gives {the corresponding result for geometric categories.}
\begin{prop}  $\gcat \chi(K)\leq \gscat K$.
\end{prop}

The next {Corollary is a direct reformulation} in categorical
terms of original results due to J. Barmak (Corollary 5.2.8 of
\cite{BARMAK}).

{\begin{cor}\label{515}{\ }
\begin{enumerate}
\item
$\cat X=0$ if and only if $\scat \K(X)=0$. In other words, $X$
is contractible if and only if its order complex $\K(X)$ is
strongly collapsible.

%

\item
$\scat K=0$ if and only if $\cat \chi(K)=0$, that is the
complex $K$ is strongly collapsible if and only if its order poset
$\chi(K)$ is contractible.
\end{enumerate}
\end{cor}}
Finally, we compare the simplicial category of a complex and of
its first {barycentric} subdivision.
\begin{cor} If $K$ is a simplicial complex, then the category of its first barycentric subdivision satisfies $\scat \sd(K) \leq \scat K$.
\end{cor}

\begin{proof} Since $K^\prime=\K(\chi(K))$ equals $\sd(K)$, it follows from Propositions \ref{AAA} and \ref{BBB} that
$$\scat K^\prime \leq  \cat \chi(K) \leq \scat K.\eqno\qedhere$$
\end{proof}

Notice that a complex $K$ and its barycentric subdivision $\sd(K)$
may not have the same strong homotopy type. For instance
\cite[Example 5.1.13]{BARMAK} if $K$ is the boundary of a
$2$-simplex then {both complexes $K$ and $\sd(K)$} do not have
beat points. Then if they were in the same homotopy class they
would be isomorphic, by Stong's result \cite[Th.~3]{STONG}. But
obviously they are not. {However, as pointed out in the proof of
\corref{515},  a complex $K$ is strong collapsible if and only if
its barycentric subdivsion $\sd(K)$ is strong collapsible. }

\section*{Acknowledgements}
The authors would like to thank Jonathan Barmak, Gabriel Minian,
John Oprea, Antonio Quintero and Daniel Tanr\'e for their valuable
comments and suggestions.

\vspace{3cm}

\address{
\noindent {\sc Desamparados Fern\'andez-Ternero}.
\\Dpto. de Geometr\'{\i}a y Topolog\'{\i}a, Universidad de Sevilla, Spain.\\}
\email{desamfer@us.es}

\medskip

\address{
\noindent {\sc Enrique Mac\'ias-Virg\'os}.
\\{Dpto. de Geometr\'{\i}a y Topolog\'{\i}a,} Universidade de San\-tia\-go de Compostela, Spain.\\}
\email{quique.macias@usc.es}

\medskip

\address{
\noindent {\sc
Jos\'e Antonio Vilches Alarc\'on}.
\\Dpto. de Geometr\'{\i}a y Topolog\'{\i}a, Universidad de Sevilla, Spain.\\}
\email{vilches@us.es}

\end{document}